\documentclass[12pt, reqno]{amsart}
\usepackage{amsmath}  \hoffset=-56pt

\usepackage{amsmath, amsfonts, rawfonts,amssymb,txfonts, enumerate}
\usepackage{eucal}
\usepackage{latexsym}
\usepackage{cite, color}
 \usepackage{bm}

\newcommand{\m}{\mathbb R^{n-1}}

\newcommand{\rnnn}{\mathbb R^n}

\newcommand{\sn}{ {\mathbb{S}^{n-1}}}
\newcommand{\R}{\mathbb R}

\newcommand{\psum}{{+_{\negthinspace\kern-2pt p}}\,}
\newcommand{\qsum}[1]{{+_{\negthinspace\kern-2pt #1}}\,}
\newcommand{\dpsum}{{\tilde+_{\negthinspace\kern-1pt p}}\,}
\newcommand{\dqsum}[1]{{\tilde+_{\negthinspace\kern-1pt #1}}\,}
\newcommand{\lsub}[1]{\hskip -1.5pt\lower.5ex\hbox{$_{#1}$}}

\numberwithin{equation}{section}

\newtheorem{theo}{Theorem}[section]

\newtheorem{lem}[theo]{Lemma}
\newtheorem{prop}[theo]{Proposition}

 \theoremstyle{definition}

\begin{document}

\title{A class of generalized fully nonlinear curvature flows and its applications}
\author[J. Hu]{Jinrong Hu}
\address{School of Mathematics, Hunan University, Changsha, 410082, Hunan Province, China}
\email{hujinrong@hnu.edu.cn}
\author[J. Liu]{Jiaqian Liu}
\address{School of Mathematics, Hunan University, Changsha, 410082, Hunan Province, China}
\email{liujiaqian@hnu.edu.cn}
\author[D. Ma]{Di Ma}
\address{School of Mathematics, Hunan University, Changsha, 410082, Hunan Province, China}
\email{madi@hnu.edu.cn}
\author[J. Wang]{Jing Wang}
\address{School of Mathematics, Hunan University, Changsha, 410082, Hunan Province, China}
\email{wangjingz@hnu.edu.cn}

\begin{abstract}
In this paper, we concern a generalized fully nonlinear curvature flow involving $k$-th elementary symmetric function for principal curvature radii in Eulidean space $\rnnn$, $k$ is an integer and $1\leq k\leq n-1$. For $1\leq k< n-1$, based on some initial data and constrains on smooth positive function defined on the unit sphere $\sn$, we obtain the long time existence and convergence of the flow.  Especially, the same result shall be derived for $k=n-1$ without any constraint on the smooth positive function.

\end{abstract}
\maketitle

\baselineskip18pt

\parskip3pt

\section{Introduction}

The geometric flows are deemed as an effective tool for promoting the research in some fields, such as geometric analysis, PDEs etc. Among a range of flows, one with the speed of functions involving principal curvatures have been widely studied. Firey\cite{F74} first introduced the Gauss curvature flow to describe shape of worn stones. Husiken \cite{Hu84} concerned with the mean curvature flow to characterise the asymptotic behaviour of the hypersurface. Chow\cite{C85}, Andrews\cite{An94}, Brendle-Choi-Dasklopoulos\cite{BCD17} considered the curvature flow at the speed of $\alpha$-power of Gauss curvature to analyse the deformation of hypersurface, for more reading, refer to \cite{An99, CW00, CHZ19, LSW20, LSW200}. Another main issue is the study of the problem on the existence of the prescribed polynomial of the principal curvature radii of the hypersurface. Urbas\cite{U91}, Gerhart\cite{G14}, Chow-Tsai\cite{CT}, Bryan-Ivaki-Scheuer\cite{BIS21} shed light on the convergence for the flow with the speed of symmetric polynomial of the principal curvature radii of the hypersurface. In this paper, we consider a class of generalized fully nonlinear curvature flow of convex hypersurfaces $M_{t}$ parameterized by smooth map $X(\cdot,t):\sn\rightarrow \rnnn$ satisfying
\begin{equation}\label{mainflow}
\left\{
\begin{array}{lr}
\frac{\partial X(x,t)}{\partial t}=\frac{1}{f(v)}\sigma_{k}(x,t)\varphi(X\cdot v)(X\cdot v)G(X)v-X;  \\
X(x,0)=X_{0}(x),
\end{array}\right.
\end{equation}
where $\sigma_{k}$ is the $k$-th elementary symmetric function for principal curvature radii,  $\varphi:(0,+\infty)\rightarrow (0,+\infty)$, $G:\rnnn \backslash 0\rightarrow (0,+\infty)$ are two given continuous functions, and $v=x$ is the unit outer normal vector of $M_{t}$ at $X(\cdot,t)$. An important feature of our flow \eqref{mainflow} is that the speed is a more general curvature function without homogeneity. If a positive self-similar solution of \eqref{mainflow} exists, it is a solution to the following fully nonlinear equation
\begin{equation}\label{Or-Mong}
\gamma\varphi(h)G(\nabla h)\sigma_{k}(x)=f(x)\quad \ {\rm for} \ \gamma=1.
\end{equation}

From the point view of convex geometry, the solvability of \eqref{Or-Mong} shall be deemed as a generalization of $L_{p}$ Christoffel-Minkowski problem. It covers many well-known results. When $k=n-1$, \eqref{Or-Mong} reduces to the smooth case of dual Orlicz-Minkowski problem studied by \cite{LY20,CL22}. In the case $\varphi(s)=s^{1-p}$, $G(y)=|y|^{q-n}$, \eqref{Or-Mong} is just the $L_{p}$ dual Minkowski problem and is the dual Minkowski problem for the index $p=1$ characterized by the dual curvature measure, which was first proposed by \cite{HLYZ16}. When $1\leq k <n-1$, for $\varphi(s)=s^{1-p}$, $G(y)=1$, \eqref{Or-Mong} is so-called $L_{p}$-Christoffel-Minkowski problem. This problem is related to the problem of prescribing $k$-th $p$-area measures, Ivaki \cite{Iva19} and Sheng-Yi\cite{SY20} gave the existence of smooth solutions in the case $p\geq k+1$ from the perspective of geometric flows.  Li-Ju-Liu\cite{LL22} treat the case $\varphi(s)=s^{1-p}$, $G(y)=|y|^{q-n}$ for $1\leq k <n-1$, and Ju-Li-Liu\cite{JL21} deal with the case $G(y)=1$ for $1\leq k <n-1$. To the best of our knowledge, the flow \eqref{mainflow} has never been studied before. Our research shall generalize and enrich the related topic to some degree.

To obtain the solvability of \eqref{Or-Mong} via the flow \eqref{mainflow}, we need some constraints on $\varphi$ and $G$.
\begin{enumerate}\label{conA}
\item[({\bf A}):] Suppose $\varphi(s)\max_{|y|=s}G(y)<\beta_{0}s^{-k-\varepsilon}$ for some positive constants $\varepsilon$, $\beta_{0}$ for $s$ near $+\infty$, and $\varphi(s)\min_{|y|=s}G(y)>\beta_{1}s^{-k-\varepsilon}$ for some positive constants $\varepsilon$, $\beta_{1}$ for $s$ near 0. Here $k$ is the order of $\sigma_{k}$.

\end{enumerate}

We are now in a position to state that the main aim of current work is to obtain the long time existence and convergence results of the flow \eqref{mainflow}. It is shown in the following theorem.
\begin{theo}\label{main*}
Let $M_{0}$ be a smooth, and strictly convex hypersurface in $\rnnn$ enclosing the origin. Supposing $1\leq k <n-1$, $k$ is an integer. Let $\varphi, G$ be smooth functions satisfying ${\bf A}$. Furthermore, for any $s>0$, $y\in \rnnn \backslash 0$,
\[
\frac{\partial }{\partial s}\left( s\frac{\partial}{\partial s}(\log \varphi(s))\right)\geq0,\quad -\vartheta \leq s\frac{\partial}{\partial s}(\log \varphi(s))\leq -1,
\]
\[
k+s\frac{\partial }{\partial s}(\log\varphi(s))+\frac{\nabla G(y)\cdot y}{G}<0,\quad \nabla G(y)\cdot y_{\tau}=0,
\]
where $\vartheta$ is a positive constant and $y_{\tau}$ is the tangential component of $y$. If $f$ is a smooth function on $\sn$ such that
\[
(k+1)f^{-\frac{1}{k+\vartheta}}\delta_{ij}+(k+\vartheta)\nabla_{ij}(f^{-\frac{1}{k+\vartheta}})
\]
is positive definite. Then there exists a smooth, strictly convex solution $M_{t}$ to flow equation \eqref{mainflow} for all time $t>0$, and it subconverges in $C^{\infty}$ to a smooth and strictly convex solution to equation \eqref{Or-Mong} for $\gamma=1$.
\end{theo}

In view of Theorem \ref{main*}, for $1\leq k < n-1$, if $\varphi(s)=s^{1-p}$, $G(y)=1$ for $p\geq k+1$, equipped with the above condition on $f$, Theorem \ref{main*} recovers the existence results to the $L_{p}$-Christoffel-Minkowski problem which have been derived by \cite{Iva19,SY20,HMS04}. It should be remarked that, for $k=n-1$, the same result shall be showed in Theorem \ref{main*} without constraint on $f$.
\begin{theo}\label{main*2}
Let $M_{0}$ be a smooth, and strictly convex hypersurface in $\rnnn$ enclosing the origin, and $f$ be a smooth and positive function on $\sn$. Supposing $k=n-1$. Let $\varphi, G$ be smooth functions satisfying ${\bf A}$. Then there exists a smooth, strictly convex solution $M_{t}$ to flow equation \eqref{mainflow} for all time $t>0$, and it subconverges in $C^{\infty}$ to a smooth and strictly convex solution to equation \eqref{Or-Mong} for $\gamma=1$.
\end{theo}

Notice that Theorem \ref{main*2} recovers the existence results to the dual Orlicz-Minkowski problem showed in \cite{LY20} from the Gauss curvature flow point of view.

The organization of this paper goes as follows: In Section \ref{Sec2}, we collect some basic knowledge about $k$-th Hessian operators and convex bodies. In Section \ref{Sec3}, we introduce the geometric flow. In Section \ref{Sec4} and \ref{Sec5}, we obtain the priori estimates of the solution to the relevant flow. In Section \ref{Sec6}, we complete the proof of Theorem \ref{main*} and Theorem \ref{main*2}. At last, we shall provide a special uniqueness result to the solution of equation \eqref{Or-Mong}.

\section{Preliminaries}
\label{Sec2}
We give some basics on $k$-Hessian operators and convex bodies, recommended to see \cite{Guan} and \cite{S14} for good references.
\subsection{$K$-th Hessian Operators}
The $k$-th elementary symmetric function for $\lambda=(\lambda_{1},\cdots,\lambda_{n-1})\in \m$ is defined as
\[
\sigma_{k}(\lambda)=\sum_{1\leq i_{1}<i_{2}<\cdots<i_{k}\leq n-1}\lambda_{i_{1}}\lambda_{i_{2}}\cdots\lambda_{i_{k}},\quad 1\leq k \leq n-1.
\]
Let $S_{n-1}$ be the set of all symmetric $(n-1)\times (n-1)$ metrics. The $k$-th elementary symmetric function for $A\in S_{n-1}$ is
\[
\sigma_{k}(A)=:\sigma_{k}(\lambda(A)),\quad \lambda(A) \ {\rm is \ the \ eigenvalue \ of} \ A.
\]
We say $A\in S_{n-1}$ belongs to $\Gamma_{k}$ if its eigenvalue $\lambda(A)\in \Gamma_{k}$.
The Garding cone $\Gamma_{k}$ is defined as
\[
\Gamma_{k}=\{\lambda\in \m|\sigma_{i}(\lambda)> 0, \ {\rm for } \ 1\leq i\leq k\}.
\]
Here we denote by $\sigma_{k}(\lambda|i)$ the symmetric function with $\lambda_{i}=0$. Now, we list some standard formulas and properities of elementary symmetric functions that we shall use in what follows.
\begin{prop}
Let $\lambda=(\lambda_{1},\cdots, \lambda_{n-1})\in \m$ and $k=0,1,\cdots, n-1$. Then,
\begin{enumerate}
\item[(i)] $\sigma_{k+1}(\lambda)=\sigma_{k+1}(\lambda|i)+\lambda_{i}\sigma_{k}(\lambda|i), \quad \forall 1\leq i\leq n-1$.
\item[(ii)] $\sum_{i=1}^{n-1}\lambda_{i}\sigma_{k}(\lambda|i)=(k+1)\sigma_{k+1}(\lambda)$.

\item[(iii)]$\sum_{i=1}^{n-1}\sigma_{k}(\lambda|i)=(n-k-1)\sigma_{k}(\lambda)$.

\item[(iv)] $\frac{\partial \sigma_{k+1}(\lambda)}{\partial \lambda_{i}}=\sigma_{k}(\lambda|i)$.

\end{enumerate}
\end{prop}
\begin{prop}[Newton-Maclaurin inequality] For $\lambda\in \Gamma_{k}$ and $1\leq l \leq k\leq n-1$, we have
\[
\left(\frac{\sigma_{k}(\lambda)}{C^{k}_{n-1}}\right)^{\frac{1}{k}}\leq \left(\frac{\sigma_{l}(\lambda)}{C^{l}_{n-1}}\right)^{\frac{1}{l}}.
\]
\end{prop}

\begin{prop}[Concavity]
For any $k>l\geq 0$, suppose $A=A_{ij}\in \Theta_{n-1}$ such that $\lambda(A)\in \Gamma_{k}$. Then, we have
\[
\left[\frac{\sigma_{k}(A)}{\sigma_{l}(A)}\right]^{\frac{1}{k-l}}\ {\rm is \ a \ concave \ function \ in } \ \Gamma_{k}.
\]

\end{prop}

\begin{prop}[Inverse Concavity] Suppose $F$ is a smooth symmetric function in $\Gamma_{n-1}$. $F$ is inverse concave, i.e.,
\[
F_{*}(\lambda_{i})=\frac{1}{F(\lambda^{-1}_{i})} \ {\rm is} \ {\rm concave}.
\]
\end{prop}

\subsection{Basics of convex bodies}
 For the good references on convex bodies, please refer to Gardner \cite{G06} and Schneider \cite{S14}.

 For $Y,Z\in {\rnnn}$, $Y\cdot Z$ denotes the standard inner product. For $X\in{\rnnn}$, we denote by $|X|=\sqrt{X\cdot X}$ the Euclidean norm. Let ${\sn}$ be the unit sphere, and $C({\sn})$ be the set of continuous functions defined on the unit sphere ${\sn}$. A compact convex set of ${\rnnn}$ with non-empty interior is called a convex body.

If $\Omega$ is a convex body containing the origin in ${\rnnn}$, for $x\in{\sn}$, the support function of $\Omega$ (with respect to the origin) is defined by
\[
h(\Omega,x)=\max\{x\cdot Y:Y \in \Omega\}.
\]
Extending this definition to a homogeneous function of degree one in $\rnnn\backslash \{0\}$ by the equation $h(\Omega, x)=|x|h\left(\Omega, \frac{x}{|x|}\right)$. The radial function $\rho$ of $\Omega$ is given by
 \[
 \rho(\Omega,u)=\max\{a>0:au\in \Omega\}, \quad u\in \sn.
 \]
 Note that $\partial \Omega$ can be represented by its radial function,
 \[
 \partial \Omega=\{\rho(\Omega,u)u:\ u\in \sn\}.
 \]
 The map $g:\partial \Omega\rightarrow {\sn}$ denotes the Gauss map of $\partial\Omega$. Meanwhile, for $\omega\subset {\sn}$, the inverse of Gauss map $g$ is expressed as
\begin{equation*}
g^{-1}(\omega)=\{X\in \partial \Omega: {\rm g(X) \ is \ defined \ and }\ g(X)\in \omega\}.
\end{equation*}
For simplicity in the subsequence, we abbreviate $g^{-1}$ as $F$. In particular, for a convex body $\Omega$ being of class $C^{2}_{+}$, i.e., its boundary is of class $C^{2}$ and of positive Gauss curvature, the support function of $\Omega$ can be written as
\begin{equation}\label{hhom}
h(\Omega,x)=x\cdot F(x)=g(X)\cdot X, \ {\rm where} \ x\in {\sn}, \ g(X)=x \ {\rm and} \ X\in \partial \Omega.
\end{equation}
 Let $e_{1},e_{2},\ldots, e_{n-1}$ be a local orthonormal frame on ${\sn}$, $h_{i}$ be the first order covariant derivatives of $h(\Omega,\cdot)$ on ${\sn}$ with respect to the frame. Differentiating \eqref{hhom} with respect to $e_{i}$ , we get
\[
h_{i}=e_{i}\cdot F(x)+x\cdot F_{i}(x).
\]
Since $F_{i}$ is tangent to $ \partial \Omega$ at $F(x)$, we obtain
\begin{equation}\label{Fi}
h_{i}=e_{i}\cdot F(x).
\end{equation}
Combining \eqref{hhom} and \eqref{Fi}, we have (see also \cite[p. 97]{U91})
\begin{equation}\label{Fdef}
F(x)= h_{i}(\Omega,x)e_{i}+h(\Omega,x)x=\nabla_{\sn}h(\Omega,x)+h(\Omega,x)x.
\end{equation}
Here $\nabla_{\sn}$ is the spherical gradient. On the other hand, since we can extend $h(\Omega,x)$ to $\rnnn$ as a 1-homogeneous function $h(\Omega, \cdot)$, then restrict the gradient of $h(\Omega,\cdot)$ on $\sn$, it yields that (see for example \cite[p. 14-16]{Guan})
\begin{equation}\label{hf}
\nabla h(\Omega,x)=F(x), \ \forall x\in{\sn},
\end{equation}
where $\nabla$ is the gradient operator in $\rnnn$. Let $h_{ij}$ be the second order covariant derivatives of $h(\Omega,\cdot)$ on ${\sn}$ with respect to the local frame. Then, applying \eqref{Fdef} and \eqref{hf}, we have (see, e.g., \cite[p. 382]{J91})
\begin{equation}\label{hgra}
\nabla h(\Omega,x)=h_{i}e_{i}+hx, \quad F_{i}(x)=(h_{ij}+h\delta_{ij})e_{j}.
\end{equation}

\section{The geometric flow and relevant functional}
\label{Sec3}
In this section, we are in the place to introduce the geometric flow and the relevant functional.

Let $M_{0}$ be a smooth, origin symmetric and strictly convex body in $\rnnn$, as presented above, we are concerned with a family of convex hypersurfaces $M_{t}$ parameterized by smooth map $X(\cdot,t):\sn\rightarrow \R$ satisfying the following flow equation,
\begin{equation}\label{xOrflow}
\left\{
\begin{array}{lr}
\frac{\partial X(x,t)}{\partial t}=\frac{1}{f(v)}\sigma_{k}(x,t)\varphi(X\cdot v)(X\cdot v)G(X)v-X(x,t);  \\
X(x,0)=X_{0}(x),
\end{array}\right.
\end{equation}
Multiplying both sides of \eqref{xOrflow} by $\nu$, by means of the definition of support function, we describe the flow equation associated with the support function $h(x,t)$ of $\Omega_{t}$ as
\begin{equation}\label{hOrflow}
\left\{
\begin{array}{lr}
\frac{\partial h(x,t)}{\partial t}=\frac{1}{f(x)}\sigma_{k}(x,t)\varphi(h(x,t))hG(\nabla_{\sn}h(x,t)+h(x,t)I)-h(x,t);  \\
h(x,0)=h_{0}(x),
\end{array}\right.
\end{equation}
Denote by $\rho(u,t)$ the radial function of $\Omega_{t}$ for any $u\in \sn$. Let $u$ and $x$ be related by
\[
\rho(u,t)u=\nabla_{\sn}h(x,t)+h(x,t)x.
\]
So,
\[
\rho(u,t)^{2}=|\nabla_{\sn}h(x,t)|^{2}+h(x,t)^{2},
\]
and $\rho(u,t)$ satisfies (see \cite{CHZ19})
\[
\frac{1}{\rho(u,t)}\frac{\partial \rho(u,t)}{\partial t}=\frac{1}{h(x,t)}\frac{\partial h(x,t)}{\partial t}.
\]
We are now in a position to give evolution equations of geometric quantities by using the flow equation \eqref{xOrflow}. For simplicity, we write
\[
\Theta=\frac{1}{f(x)}h(x,t)\varphi(h(x,t))G(\nabla_{\sn}h(x,t)+h(x,t)I), \quad P=\Theta \sigma_{k}(x,t), \quad w_{ij}=h_{ij}+h\delta_{ij}.
\]
Notice that the eigenvalue of $\{w_{ij}\}$ and $\{w^{ij}\}$ are respectively the principal radii and principal curvatures of $\partial\Omega_{t}$ (see \cite{U91}), where $\{w^{ij}\}$ is the inverse matrix of $\{w_{ij}\}$.

\begin{lem}
The following evolution equations hold along the flow \eqref{xOrflow}.
\begin{equation}
\begin{split}
\label{e1}
&\partial_{t}w_{ij}-\Theta\sigma^{pq}_{k}\nabla_{pq}w_{ij}\\
&=(k+1)\Theta\sigma_{k}\delta_{ij}-\Theta\sigma^{pq}_{k}\delta_{pq}w_{ij}+\Theta(\sigma^{ip}_{k} w_{jp}-\sigma^{jp}_{k}w_{ip})\\
&\quad +\Theta \sigma^{pq,mn}_{k}\nabla_{j}w_{pq}\nabla_{i}w_{mn}+\sigma_{k}\nabla_{ij}\Theta+\nabla_{j}\sigma_{k}\nabla_{i}\Theta+\nabla_{i}\sigma_{k}\nabla_{j}\Theta-w_{ij},
\end{split}
\end{equation}

\begin{equation}
\begin{split}
\label{e2}
&\partial_{t}w^{ij}-\Theta\sigma^{pq}_{k}\nabla_{pq}w^{ij}\\
&=-(k+1)\Theta\sigma_{k}w^{ip}w^{jp}+\Theta\sigma^{pq}_{k}\delta_{pq}w^{ij}-\Theta w^{ip}w^{jq}(\sigma^{rp}_{k}w_{rq}-\sigma^{rq}_{k}w_{rp})\\
&\quad -\Theta w^{il}w^{js}(\sigma^{pq,mn}_{k}+2\sigma^{pm}_{k}w^{nq})\nabla_{l}w_{pq}\nabla_{s}w_{mn}\\
&\quad -w^{ip}w^{jq}(\sigma_{k}\nabla_{pq}\Theta+\nabla_{q}\sigma_{k}\nabla_{p}\Theta+\nabla_{p}\sigma_{k}\nabla_{q}\Theta)+w^{ij},
\end{split}
\end{equation}
and
\begin{equation}
\begin{split}
\label{e3}
&\partial_{t}\left(\frac{\rho(u,t)^{2}}{2}\right)-\Theta\sigma^{ij}_{k}\nabla_{ij}\left(\frac{\rho(u,t)^{2}}{2}\right)\\
&=(k+1)h\Theta\sigma_{k}-\rho^{2}+\sigma_{k}\nabla_{i}h\nabla_{j}\Theta-\Theta\sigma^{ij}_{k}w_{mi}w_{mj}.
\end{split}
\end{equation}
\end{lem}
\begin{proof}
For the specific computations, one can refer to \cite{Iva19}. Here for reader's convenience, we give the detail. For \eqref{e1}, since
\[
\partial_{t}\nabla_{ij}h=\nabla_{ij}(\partial_{t}h)=\sigma_{k}\nabla_{ij}\Theta+\nabla_{j}\sigma_{k}\nabla_{i}\Theta+\nabla_{i}\sigma_{k}\nabla_{j}\Theta+\Theta\nabla_{ij}\sigma_{k}-h_{ij},
\]
where
\[
\nabla_{ij}\sigma_{k}=\sigma^{pq,mn}_{k}\nabla_{j}w_{pq}\nabla_{i}w_{mn}+\sigma^{pq}_{k}\nabla_{ij}w_{pq}.
\]
By virtue of Ricci identity,
\[
\nabla_{ij}w_{pq}=\nabla_{pq}w_{ij}+\delta_{ij}\nabla_{pq}w-\delta_{pq}\nabla_{ij}w+\delta_{iq}\nabla_{pj}w-\delta_{pj}\nabla_{iq}w.
\]
So,
\begin{equation*}
\begin{split}
\label{}
\partial_{t}h_{ij}&=\Theta\sigma^{pq}_{k}\nabla_{pq}w_{ij}+k\Theta\sigma_{k}\delta_{ij}-\Theta\sigma^{pq}_{k}\delta_{pq}w_{ij}+\Theta(\sigma^{ip}_{k}w_{jp}-\sigma^{jq}_{k}w_{iq})\\
&\quad +\Theta\sigma^{pq,mn}\nabla_{j}w_{pq}\nabla_{i}w_{mn}+\sigma_{k}\nabla_{ij}\Theta+\nabla_{j}\sigma_{k}\nabla_{i}\Theta+\nabla_{i}\sigma_{k}\nabla_{j}\Theta-h_{ij}.
\end{split}
\end{equation*}
This together with $w_{ij}=h_{ij}+h\delta_{ij}$ gives \eqref{e1}. Then using \eqref{e1}, the evolution equation \eqref{e2} of $w^{ij}$ follows from
\[
\partial_{t}w^{ij}=-w^{im}w^{lj}\partial_{t}w_{ml},\quad \nabla_{pq}w^{ij}=2w^{ik}w^{mn}w_{knq}w^{lj}w_{mlp}-w^{im}w^{lj}\nabla_{pq}w_{ml}.
\]
For \eqref{e3}, due to $\rho^{2}=h^{2}+|\nabla_{\sn}h|^{2}$, we directly compute
\begin{equation}
\begin{split}
\label{}
&\partial_{t}\left(\frac{\rho^{2}}{2} \right)-\Theta\sigma^{ij}_{k}\nabla_{ij}\left(\frac{\rho^{2}}{2} \right)\\
&=h\partial_{t}h+\nabla_{i}h\nabla_{i}\partial_{t}h-\Theta\sigma^{ij}_{k}(h\nabla_{ij}h+\nabla_{i}h\nabla_{j}h+\nabla_{m}h\nabla_{j}\nabla_{mi}h+\nabla_{mi}h\nabla_{mj}h)\\
&=h\partial_{t}h+\nabla_{i}h\nabla_{i}(\Theta\sigma_{k}-h)-\Theta\sigma^{ij}_{k}[\nabla_{i}h\nabla_{j}h+\nabla_{m}h\nabla_{j}(w_{mi}-h\delta_{mi})]\\
&\quad -\Theta\sigma^{ij}_{k}h(w_{ij}-h\delta_{ij})-\Theta\sigma^{ij}_{k}(w_{mi}-h\delta_{mi})(w_{mj}-h\delta_{mj})\\
&=(k+1)h\Theta\sigma_{k}-\rho^{2}+\sigma_{k}\nabla_{i}h\nabla_{i}\Theta-\Theta\sigma^{ij}_{k}w_{mi}w_{mj}.
\end{split}
\end{equation}
Hence, the proof is completed.
\end{proof}

\section{$C^{0},C^{1}$ estimates}
\label{Sec4}
In this section, we shall obtain the $C^{0}$, $C^{1}$ estimates of solutions to the flow \eqref{xOrflow}. Let us begin with completing the $C^{0}$ estimate.

\begin{lem}\label{C0}
Supposing $1\leq k \leq n-1$, $k$ is an integer. Let $f$ be a smooth and positive function on $\sn$, $M_{t}$ be a smooth and strictly convex solution satisfying the flow \eqref{xOrflow}, and let $\varphi, G$ be smooth functions satisfying ${\bf A}$. Then
\begin{equation}\label{C0*}
\frac{1}{C}\leq h(x,t)\leq C, \ \forall (x,t)\in {\sn}\times (0,\infty),
\end{equation}
and
\begin{equation}\label{C00*}
\frac{1}{C}\leq \rho(u,t)\leq C, \ \forall (u,t)\in {\sn}\times (0,\infty).
\end{equation}
Here $h(x,t)$ and $\rho(u,t)$ are the support function and the radial function of $M_{t}$.
\end{lem}
\begin{proof}
Suppose the spatial maximum of $h(x,t)$ is attained at $x^{0}\in \sn$. Then at $x^{0}$, we have
\[
\nabla_{\sn}h=0, \ \nabla^{2}_{\sn}h\leq 0,\  \rho=h,
\]
and
\[
\nabla^{2}_{\sn}h+hI\leq hI.
\]
This illustrates that, at $x^{0}$,
\begin{equation*}
\begin{split}
\label{}
\frac{\partial h }{\partial t}&=\frac{1}{f(x)}h\sigma_{k}G(\rho,u)\varphi(h)-h\\
&\leq \frac{1}{f_{\min}}h[G(h,u)h^{k}\varphi(h)-f_{\min}].\\
\end{split}
\end{equation*}
If $\varphi(s)\max_{|y|=s}G(y)< \beta_{0} s^{-k-\varepsilon}$ for some positive constant $\varepsilon$ and $\beta_{0}$ for $s$ tends to $+\infty$, then, at $x^{0}$, there is
\begin{equation}
\begin{split}
\label{hmax1}
\frac{\partial h}{\partial t}\leq \frac{1}{f_{\min}}h[h^{-\varepsilon}\beta_{0}-f_{\min}].
\end{split}
\end{equation}
The right hand side of \eqref{hmax1} is negative for $\max_{\sn}h(x,t)$ large enough, thus the upper bound of $\max_{\sn}h(x,t)$ follows.

Similarly, we can estimate the spatial minimum of $h(x,t)$. Suppose $\min_{\sn}h(x,t)$ is attained at $x^{1}$. Then at $x^{1}$, we have
\[
\nabla_{\sn}h=0, \ \nabla^{2}_{\sn}h\geq 0, \  \rho=h,
\]
and
\[
\nabla^{2}_{\sn}h+hI\geq hI.
\]
Therefore, at $x^{1}$, we have
\begin{equation*}
\begin{split}
\label{}
\frac{\partial h }{\partial t}&=\frac{1}{f(x)}h\sigma_{k}G(\rho,u)\varphi(h)-h\\
&\geq \frac{1}{f_{\max}}h[G(h,u)h^{k}\varphi(h)-f_{\max}].\\
\end{split}
\end{equation*}
If $\varphi(s)\min_{|y|=s}G(y)> \beta_{1} s^{-k-\varepsilon}$ for some positive constants $\varepsilon$ and $\beta_{1}$ for $s$ near 0, then, at $x^{1}$, we have
\begin{equation}
\begin{split}
\label{hmax}
\frac{\partial h}{\partial t}\geq \frac{1}{f_{\max}}h[h^{-\varepsilon}\beta_{1}-f_{\max}].
\end{split}
\end{equation}
The right hand side of \eqref{hmax} is positive for $\min_{\sn}h(x,t)$ small enough and the lower bound of $\min_{\sn}h(x,t)$ follows. Due to $\rho(u,t)u=\nabla_{\sn} h(x,t)+h(x,t)x$, one has (see \cite{CL21})
\begin{equation}\label{p1}
\min_{\sn} h(x,t)\leq \rho(u,t)\leq \max_{\sn}h(x,t).
\end{equation}
Hence, the estimate \eqref{C00*} immediately holds via \eqref{p1}. The proof is completed.
\end{proof}
The $C^{1}$ estimate is as follows.
\begin{lem}\label{C1}
Supposing $1\leq k \leq n-1$, $k$ is an integer. Let $f$ be a smooth and positive function on $\sn$, $M_{t}$ be a smooth and strictly convex solution satisfying the flow \eqref{xOrflow}, and let $\varphi, G$ be smooth functions satisfying ${\bf A}$. Then
\begin{equation*}\label{C1*}
|\nabla_{\sn} h(x,t)|\leq C, \ \forall (x,t)\in {\sn}\times (0,\infty),
\end{equation*}
and
\begin{equation*}\label{C11*}
|\nabla_{\sn} \rho(u,t)|\leq C, \ \forall (u,t)\in {\sn}\times (0,\infty)
\end{equation*}
for some $C> 0$, independent of $t$ .
\end{lem}
\begin{proof}
 Lemma \ref{C0} and the following facts
\[
\rho^{2}=h^{2}+|\nabla_{\sn} h|^{2},  \quad h=\frac{\rho^{2}}{\sqrt{\rho^{2}+|\nabla_{\sn} \rho|^{2}}}
\]
 elaborate the results.
\end{proof}
\section{$C^{2}$ estimate}
\label{Sec5}
In this section, our aim is to derive the $C^{2}$ estimate of solutions to the flow \eqref{xOrflow}. We first obtain the lower bound of $\sigma_{k}$.

\begin{lem}\label{prin}
Supposing $1\leq k \leq n-1$, $k$ is an integer. Let $f$ be a smooth and positive function on $\sn$, $M_{t}$ be a smooth and strictly convex solution satisfying the flow \eqref{xOrflow}, and let $\varphi, G$ be smooth functions satisfying ${\bf A}$. Then
\[
\sigma_{k}(x,t)\geq C,\ \forall (x,t)\in {\sn}\times (0,\infty),
\]
for a positive constant $C$, independent of $t$.
\end{lem}

\begin{proof}
Considering the auxiliary function
\begin{equation*}\label{}
Q=\log P-A\frac{\rho^{2}}{2},
\end{equation*}
where $P=\Theta\sigma_{k}$, and $A$ is a positive constant to be determined later. Assume that the spatial minimum of $Q$ is achieved at $\tilde{x}_{0}$. Then at $\tilde{x}_{0}$, we have
\begin{equation}
\begin{split}
\label{Qi}
0=\nabla_{i}Q=\frac{\nabla_{i}P}{P}-A\nabla_{i}\left(\frac{\rho^{2}}{2}\right),
\end{split}
\end{equation}
and
\begin{equation*}
\begin{split}
\label{}
0\leq \nabla_{ij}Q=\frac{\nabla_{ij}P}{P}-\frac{\nabla_{i}P\nabla_{j}P}{P^{2}}-A\nabla_{ij}\left(\frac{\rho^{2}}{2}\right).
\end{split}
\end{equation*}
Now, we compute
\begin{equation}
\begin{split}
\label{Qt}
&\partial_{t}Q-\Theta\sigma^{ij}_{k}\nabla_{ij}Q\\
&=\frac{1}{P}(\partial_{t}P-\Theta\sigma^{ij}_{k}\nabla_{ij}P)-A\left[\partial_{t}\left(\frac{\rho^{2}}{2}\right)-\Theta\sigma^{ij}_{k}\nabla_{ij}\left(\frac{\rho^{2}}{2}\right)\right]+\frac{\Theta}{P^{2}}\sigma^{ij}_{k}\nabla_{i}P\nabla_{j}P.
\end{split}
\end{equation}
Clearly
\begin{equation}
\begin{split}
\label{Pt}
\partial_{t}P=\Theta \partial_{t}\sigma_{k}+\sigma_{k}\partial_{t}\Theta,
\end{split}
\end{equation}
in which
\begin{equation*}
\begin{split}
\label{}
\partial_{t}\sigma_{k}&=\sigma^{ij}_{k}\partial_{t}(\nabla_{ij}h+\delta_{ij}h)\\
&=\sigma^{ij}_{k}\nabla_{ij}P-\sigma^{ij}_{k}\nabla_{ij}h+\sigma^{ij}_{k}\delta_{ij}P-\sigma^{ij}_{k}\delta_{ij}h\\
&=\sigma^{ij}_{k}\nabla_{ij}P-\sigma^{ij}_{k}(w_{ij}-h\delta_{ij})+\sigma^{ij}_{k}\delta_{ij}P-\sigma^{ij}_{k}\delta_{ij}h\\
&=\sigma^{ij}_{k}\nabla_{ij}P+\sigma^{ij}_{k}\delta_{ij}P-k\sigma_{k},
\end{split}
\end{equation*}
and
\begin{equation*}
\begin{split}
\label{}
\partial_{t}\Theta&=\frac{1}{f}G\varphi\frac{\partial h}{\partial t}+\frac{1}{f}h\varphi\left(\nabla G\cdot \frac{\partial X}{\partial t}\right)+\frac{1}{f}hG\varphi{'}\frac{\partial h}{\partial t}\\
&=\frac{1}{f}G\varphi(P-h)+\frac{1}{f}h\varphi(\nabla G\cdot (Pv-X))+\frac{1}{f}hG\varphi{'}(P-h).
\end{split}
\end{equation*}
In turn, \eqref{Pt} becomes
\begin{equation*}
\begin{split}
\label{}
\partial_{t}P&=\Theta(\sigma^{ij}_{k}\nabla_{ij}P+\sigma^{ij}_{k}\delta_{ij}P-k\sigma_{k})+\sigma_{k}\left[\frac{1}{f}G\varphi(P-h)+\frac{1}{f}h\varphi(\nabla G\cdot (Pv-X))+\frac{1}{f}hG\varphi{'}(P-h)\right]\\
&=\Theta\sigma^{ij}_{k}\nabla_{ij}P+\Theta\sigma^{ij}_{k}\delta_{ij}P-kP+\frac{P}{h}(P-h)+P^{2}\frac{\nabla G\cdot v}{G}-P\frac{\nabla G\cdot X}{G}+P(P-h)\frac{\varphi^{'}}{\varphi}\\
&=\Theta\sigma^{ij}_{k}\nabla_{ij}P+\Theta\sigma^{ij}_{k}\delta_{ij}P-P\left(k+1+\frac{\varphi^{'}h}{\varphi}+\frac{\nabla G\cdot X}{G}\right)+\frac{P^{2}}{h}\left(1+\frac{\varphi^{'}h}{\varphi}+\frac{(\nabla G\cdot hv)}{G}\right).
\end{split}
\end{equation*}
So, \eqref{Qt} turns into
\begin{equation}
\begin{split}
\label{klow}
&\partial_{t}Q-\Theta\sigma^{ij}_{k}\nabla_{ij}Q\\
&=\frac{\Theta}{P^{2}}\sigma^{ij}_{k}\nabla_{i}P\nabla_{j}P+\Theta\sigma^{ij}_{k}\delta_{ij}-\left(k+1+\frac{\varphi^{'}h}{\varphi}+\frac{\nabla G\cdot X}{G}\right)\\
&\quad +\frac{P}{h}\left(1+\frac{\varphi^{'}h}{\varphi}+\frac{h(\nabla G\cdot v)}{G}\right)-A(k+1)h\Theta\sigma_{k}+A\rho^{2}-A\sigma_{k}\nabla_{i}h\nabla_{i}\Theta+A\Theta\sigma^{ij}_{k}w_{mi}w_{mj}\\
&\geq A\rho^{2}-\left(k+1+\frac{\varphi^{'}h}{\varphi}+\frac{(\nabla G\cdot X)}{G}\right)+\frac{P}{h}\left(1+\frac{\varphi^{'}h}{\varphi}+\frac{(\nabla G\cdot hv)}{G}\right)-A(k+1)h\Theta\sigma_{k}-A\sigma_{k}\nabla_{i}h\nabla_{i}\Theta.
\end{split}
\end{equation}
By using the $C^{0}$, $C^{1}$ estimates, choosing $A$ large enough, the right hand side of \eqref{klow} is strictly positive provided $\min_{\sn}\sigma_{k}\rightarrow 0$ ($Q$ is negatively large enough), thus the lower bound of $Q$ follows, hence $\sigma_{k}$ is uniformly bounded below away from zero.

The upper bound of $\sigma_{k}$ is as follows.
\begin{lem}\label{prin2}
Supposing $1\leq k \leq n-1$, $k$ is an integer. Let $f$ be a smooth and positive function on $\sn$, $M_{t}$ be a smooth and strictly convex solution satisfying the flow \eqref{xOrflow}, and let $\varphi, G$ be smooth functions satisfying ${\bf A}$. Then
\[
\sigma_{k}(x,t)\leq \bar{C},\ \forall (x,t)\in {\sn}\times (0,\infty),
\]
for a positive constant $\bar{C}$, independent of $t$.
\end{lem}
\begin{proof}
Using the $C^{0}$ estimate, there exists a positive constant $B$ such that
\[
B<\rho^{2}<\frac{1}{B}
\]
for all $t>0$. Now, setting the auxiliary function as
\[
\chi(x,t)=\frac{\frac{1}{f}G\varphi\sigma_{k}}{1-\frac{B\rho^{2}}{2}}=\frac{P}{h}\frac{1}{1-\frac{B\rho^{2}}{2}}.
\]
Suppose the spatial maximum of $\chi(x,t)$ at $\tilde{x}_{1}$. Then at $\tilde{x}_{1}$, we get
\begin{equation*}
\begin{split}
\label{}
0=\nabla_{i}\chi=\nabla_{i}\left(\frac{P}{h}\right)\frac{1}{1-\frac{B\rho^{2}}{2}}+\frac{P}{h}\frac{B\nabla_{i}(\frac{\rho^{2}}{2})}{(1-\frac{B\rho^{2}}{2})^{2}},
\end{split}
\end{equation*}
and
\begin{equation*}
\begin{split}
\label{}
0\geq \nabla_{ij}\chi&=\nabla_{ij}\left(\frac{P}{h}\frac{1}{1-\frac{B\rho^{2}}{2}}\right)\\
&=\frac{1}{1-\frac{B\rho^{2}}{2}}\nabla_{ij}\left(\frac{P}{h}\right)+2\frac{B\nabla_{i}\frac{\rho^{2}}{2}}{(1-\frac{B\rho^{2}}{2})^{2}}\nabla_{j}\frac{P}{h}+2\frac{\frac{P}{h}B^{2}\nabla_{j}\frac{\rho^{2}}{2}\nabla_{i}\frac{\rho^{2}}{2}}{(1-\frac{B\rho^{2}}{2})^{3}}\\
&\quad +\frac{P}{h}\frac{B}{(1-\frac{B\rho^{2}}{2})^{2}}\nabla_{ij}\left(\frac{\rho^{2}}{2}\right)\\
&=\frac{1}{1-\frac{B\rho^{2}}{2}}\nabla_{ij}\left(\frac{P}{h}\right)+\frac{P}{h}\frac{B}{(1-\frac{B\rho^{2}}{2})^{2}}\nabla_{ij}\left(\frac{\rho^{2}}{2}\right).
\end{split}
\end{equation*}
We are in a position to compute the following
\begin{equation}
\begin{split}
\label{rit}
&\partial_{t}\chi-\Theta\sigma^{ij}_{k}\nabla_{ij}\chi\\
&=\frac{1}{1-\frac{B\rho^{2}}{2}}\partial_{t}\left(\frac{P}{h}\right)+\frac{P}{h}\frac{B}{\left(1-\frac{B\rho^{2}}{2}\right)^{2}}\partial_{t}\left(\frac{\rho^{2}}{2}\right)-\Theta\sigma^{ij}_{k}\frac{1}{1-\frac{B\rho^{2}}{2}}\nabla_{ij}\left(\frac{P}{h}\right)-\Theta\sigma^{ij}_{k}\frac{P}{h}\frac{B}{(1-\frac{B\rho^{2}}{2})^{2}}\nabla_{ij}\left(\frac{\rho^{2}}{2}\right)\\
&=\frac{1}{1-\frac{B\rho^{2}}{2}}\left[ \partial_{t}\left(\frac{P}{h}\right)-\Theta\sigma^{ij}_{k}\nabla_{ij}\left(\frac{P}{h}\right) \right]+\frac{P}{h}\frac{B}{\left(1-\frac{B\rho^{2}}{2}\right)^{2}}\left[\partial_{t}\left(\frac{\rho^{2}}{2}\right)-\Theta\sigma^{ij}_{k}\nabla_{ij}\left(\frac{\rho^{2}}{2}\right) \right],
\end{split}
\end{equation}
in which
\begin{equation}
\begin{split}
\label{inti}
&\partial_{t}\left(\frac{P}{h}\right)-N\sigma^{ij}_{k}\nabla_{ij}\left(\frac{P}{h}\right)\\
&=\frac{1}{h}(\partial_{t}P-\Theta\sigma^{ij}_{k}\nabla_{ij}P)-\frac{P}{h^{2}}(\partial_{t}h-\Theta\sigma^{ij}_{k}\nabla_{ij}h)+2\frac{P}{h}\sigma^{ij}_{k}\nabla_{i}\frac{P}{h}\nabla_{j}h\\
&=\frac{1}{h}\left[\Theta P\sigma^{ij}_{k}\delta_{ij}-P\left(k+1+\frac{\varphi^{'}}{\varphi}h+\frac{\nabla G\cdot X}{G}\right)\right]+\frac{P^{2}}{h^{2}}\left(1+\frac{\varphi^{'}h}{\varphi}+\frac{(\nabla G\cdot hv)}{G}\right)\\
&\quad -\frac{P}{h^{2}}(P-h-k\Theta\sigma_{k}+\Theta \sigma^{ij}_{k}h\delta_{ij})+2\frac{P}{h}\sigma^{ij}_{k}\nabla_{i}h\nabla_{j}\frac{P}{h}\\
&=-\frac{P}{h}\left(k+\frac{\varphi^{'}h}{\varphi}+\frac{\nabla G\cdot X}{G}\right)+\frac{P^{2}}{h^{2}}\left(k+\frac{\varphi^{'}h}{\varphi}+\frac{(\nabla G\cdot hv)}{G}\right)+2\frac{P}{h}\sigma^{ij}_{k}\nabla_{i}h\nabla_{j}\frac{P}{h}.
\end{split}
\end{equation}
In turn, \eqref{rit} becomes
\begin{equation}
\begin{split}
\label{rit2}
&\partial_{t}\chi-\Theta\sigma^{ij}_{k}\nabla_{ij}\chi\\
&=\frac{1}{1-\frac{B\rho^{2}}{2}}\left[  -\frac{P}{h}\left(k+\frac{\varphi^{'}h}{\varphi}+\frac{\nabla G\cdot X}{G}\right)+\frac{P^{2}}{h^{2}}\left(k+\frac{\varphi^{'}h}{\varphi}+\frac{h(\nabla G\cdot v)}{G}\right)+2\frac{P}{h}\sigma^{ij}_{k}\nabla_{i}h\nabla_{j}\frac{P}{h} \right]\\
&\quad +\frac{P}{h}\frac{B}{(1-\frac{B\rho^{2}}{2})^{2}}\left[(k+1)h\Theta\sigma_{k}-\rho^{2}+\sigma_{k}\nabla_{i}h\nabla_{i}\Theta-\Theta\sigma^{ij}_{k}w_{mi}w_{mj} \right].
\end{split}
\end{equation}
On one hand, at $\tilde{x}_{1}$, there satisfies
\begin{equation}\label{}
\nabla_{j}\frac{P}{h}=-\frac{P}{h}\frac{B\nabla_{j}\frac{\rho^{2}}{2}}{1-\frac{B\rho^{2}}{2}}=-\frac{P}{h}\frac{Bw_{jm}h_{m}}{1-\frac{B\rho^{2}}{2}}.
\end{equation}
On the other hand, by the inverse concavity of $(\sigma_{k})^{\frac{1}{k}}$ (see \cite{AM13,U91}) , we have
\[
\left( (\sigma_{k})^\frac{1}{k} \right)^{ij}w_{im}w_{jm}\geq (\sigma_{k})^{\frac{2}{k}},
\]
which illustrates that
\begin{equation*}\label{}
\sigma^{ij}_{k}w_{im}w_{jm}\geq k(\sigma_{k})^{1+\frac{1}{k}}.
\end{equation*}
So,
\begin{equation*}
\begin{split}
\label{}
&\partial_{t}\chi-\Theta \sigma^{ij}_{k}\nabla_{ij}\chi\\
&\leq \frac{1}{1-\frac{B\rho^{2}}{2}}\left[  -\frac{P}{h}\left(k+\frac{\varphi^{'}h}{\varphi}+\frac{\nabla G\cdot X}{G}\right)+\frac{P^{2}}{h^{2}}\left(k+\frac{\varphi^{'}h}{\varphi}+\frac{h(\nabla G\cdot v)}{G}\right) \right]\\
&\quad +\frac{P}{h}\frac{B}{(1-\frac{B\rho^{2}}{2})^{2}}\left[(k+1)h\Theta\sigma_{k}-\rho^{2}+\sigma_{k}\nabla_{i}h\nabla_{i}\Theta-\Theta (\sigma_{k})^{1+\frac{1}{k}} \right].
\end{split}
\end{equation*}
This in conjunction with the $C^{0}, C^{1}$ estimates, we have
\begin{equation}\label{}
\partial_{t}\chi\leq c_{1}\chi + c_{2}\chi^{2}- c_{3}\chi^{2+\frac{1}{k}}
\end{equation}
for some positive constants $c_{1},c_{2},c_{3}$. Then, we see that $\chi(x,t)$ is uniformly  bounded from above. Hence, the upper bound of $\sigma_{k}$ is obtained.
\end{proof}
We are in a position to derive the upper bound of the principal curvatures for $1\leq k < n-1$.
\begin{lem}\label{gusk}
Supposing $1\leq k < n-1$, $k$ is an integer.  Let $M_{t}$ be a smooth and strictly convex solution satisfying the flow \eqref{xOrflow}. Suppose $f, \varphi, G$ be smooth and positive functions satisfying the assumptions of Theorem \ref{main*} and ${\bf A}$.  Then
\begin{equation}
\kappa_{i}(x,t)\leq \hat{C}, \quad \forall(x,t)\in \sn\times(0,\infty),
\end{equation}
where $\hat{C}$ is a positive constant independent of $t$.
\end{lem}
\begin{proof}
Suppose the spatial maximum of the maximum eigenvalue of the matrix $\{\frac{w^{ij}}{h}\}$ is attained at $\tilde{x}_{2}\in \sn$. By a rotation, we may assume $\{w_{ij}(\tilde{x}_{2},t)\}$ is diagonal, and the maximum eigenvalue of the matrix $\{\frac{w^{ij}}{h}\}$ at $(\tilde{x}_{2},t) $ is $\frac{w^{11}}{h}(\tilde{x}_{2},t)$. So, at $\tilde{x}_{2}$, we have
\begin{equation*}\label{}
\nabla_{i}\left(\frac{w^{11}}{h}\right)=0,
\end{equation*}
i.e.,
\begin{equation}\label{hi}
w^{11}w_{11i}=-\frac{h_{i}}{h}.
\end{equation}
Moreover,
\begin{equation*}\label{}
\nabla_{ij}\left(\frac{w^{11}}{h}\right)\leq0.
\end{equation*}
Now, we compute the evolution equation of $\frac{w^{11}}{h}$ as
\begin{equation*}
\begin{split}
\label{}
&\partial_{t}\frac{w^{11}}{h}-\Theta\sigma^{ij}_{k}\nabla_{ij}\frac{w^{11}}{h}\\
&=\frac{2}{h}\Theta\sigma^{ij}_{k}\nabla_{i}\frac{w^{11}}{h}\nabla_{j}h+\frac{\Theta}{h^{2}}w^{11}\sigma^{ij}_{k}\nabla_{ij}h-(k+1)\frac{\Theta}{h}\sigma_{k}(w^{11})^{2}+\frac{\Theta}{h}\sigma^{ij}_{k}\delta_{ij}w^{11}\\
&\quad -\frac{\Theta}{h}(w^{11})^{2}(\sigma^{ij,mn}_{k}+2\sigma^{im}_{k}w^{nj})\nabla_{1}w_{ij}\nabla_{1}w_{mn}\\
&\quad -\frac{1}{h}(w^{11})^{2}(\nabla_{11}\Theta \sigma_{k}+2\nabla_{1}\sigma_{k}\nabla_{1}\Theta)-\frac{w^{11}}{h^{2}}\Theta\sigma_{k}+2\frac{w^{11}}{h}\\
&=\frac{2}{h}\Theta\sigma^{ij}_{k}\nabla_{i}\frac{w^{11}}{h}\nabla_{j}h-(k+1)\frac{\Theta}{h}\sigma_{k}(w^{11})^{2}-\frac{\Theta}{h}(w^{11})^{2}(\sigma^{ij,mn}_{k}+2\sigma^{im}_{k}w^{nj})\nabla_{1}w_{ij}\nabla_{1}w_{mn}\\
&\quad -\frac{1}{h}(w^{11})^{2}(\nabla_{11}\Theta \sigma_{k}+2\nabla_{1}\sigma_{k}\nabla_{1}\Theta)+(k-1)\frac{w^{11}}{h^{2}}\Theta\sigma_{k}+\frac{2w^{11}}{h}.
\end{split}
\end{equation*}
By means of the inverse concavity of $(\sigma_{k})^{\frac{1}{k}}$, there is (see \cite{AM13,U91}),
\begin{equation}\label{c1}
(\sigma_{k}^{ij,mn}+2\sigma^{im}_{k}w^{nj})\nabla_{1}w_{ij}\nabla_{1}w_{mn}\geq \frac{k+1}{k}\frac{(\nabla_{1}\sigma_{k})^{2}}{\sigma_{k}}.
\end{equation}
Utilizing Schwartz inequality, we obtain
\begin{equation}\label{c2}
2|\nabla_{1}\sigma_{k}\nabla_{1}\Theta|\leq \frac{k+1}{k}\frac{\Theta(\nabla_{1}\sigma_{k})^{2}}{\sigma_{k}}+\frac{k}{k+1}\frac{\sigma_{k}(\nabla_{1}\Theta)^{2}}{\Theta}.
\end{equation}
Using \eqref{c1} and \eqref{c2}, at $\tilde{x}_{2}$, there is
\begin{equation}\label{c3}
\partial_{t}\frac{w^{11}}{h}\leq -\frac{(w^{11})^{2}}{h}\sigma_{k}\left[ \nabla_{11}\Theta-\frac{k}{k+1} \frac{(\nabla_{1}\Theta)^{2}}{\Theta}+(k+1)\Theta+(1-k)\frac{\Theta w_{11}}{h}\right]+\frac{2w^{11}}{h}.
\end{equation}
To estimate \eqref{c3}. Let $\iota$ be the arc-length of the great circle passing through $x_{0}$ with the unit tangent vector $e_{1}$, then
\begin{equation}\label{n11}
\nabla_{11}\Theta-\frac{k}{k+1}\frac{(\nabla_{1}\Theta)^{2}}{\Theta}+(k+1)\Theta=(k+1)\Theta^{\frac{k}{k+1}}\left(\Theta^{\frac{1}{k+1}}+(\Theta^{\frac{1}{k+1}})_{\iota\iota}\right)=(k+1)\Theta(1+\Theta^{-\frac{1}{k+1}}\Theta^{\frac{1}{k+1}}_{\iota \iota}).
\end{equation}
On one hand,
\begin{equation}\label{}
\Theta_{\iota}=(f^{-1})_{\iota}\varphi hG+f^{-1}G\varphi h_{\iota}\left(1+\frac{\varphi^{'}}{\varphi}h\right)+f^{-1}\varphi h G_{\iota}.
\end{equation}
On the other hand,
\begin{equation}
\begin{split}
\label{}
\Theta_{\iota \iota}&=(f^{-1})_{\iota\iota}\varphi h G+2(f^{-1})_{\iota}\varphi Gh_{\iota}\left(1+\frac{\varphi^{'}}{\varphi}h\right)+2(f^{-1})_{\iota}\varphi hG_{\iota}+f^{-1}\varphi^{'}h^{2}_{\iota}G\left(1+\frac{\varphi^{'}}{\varphi}h\right)\\
&\quad +f^{-1}\varphi G_{\iota}h_{\iota}\left(1+\frac{\varphi^{'}}{\varphi}h\right)+f^{-1}\varphi Gh_{\iota \iota}\left(1+\frac{\varphi^{'}}{\varphi}h\right)+f^{-1}\varphi Gh^{2}_{\iota}\left(1+\frac{\varphi^{'}}{\varphi}h\right)^{'}\\
&\quad+f^{-1}\varphi^{'}hh_{\iota}G_{\iota}+f^{-1}\varphi h_{\iota}G_{\iota}+f^{-1}\varphi hG_{\iota \iota}.
\end{split}
\end{equation}
In view of \eqref{n11},
\begin{equation}
\begin{split}
\label{Gd}
&1+\Theta^{-\frac{1}{k+1}}(\Theta^{\frac{1}{k+1}})_{\iota \iota}\\
&=1+\frac{1}{k+1}\Theta^{-1}\Theta_{\iota\iota}-\frac{k}{(k+1)^{2}}\Theta^{-2}\Theta^{2}_{\iota}\\
&=1+\frac{1}{k+1}f(f^{-1})_{\iota\iota}+\frac{2f}{(k+1)h}(f^{-1})_{\iota}h_{\iota}\left(1+\frac{\varphi^{'}}{\varphi}h\right)+\frac{\varphi^{'}}{(k+1)\varphi h}h^{2}_{\iota}\left(1+\frac{\varphi^{'}h}{\varphi }\right)\\
&\quad +\frac{h_{\iota\iota}}{(k+1)h}\left(1+\frac{\varphi^{'}}{\varphi}h\right)+\frac{h^{2}_{\iota}}{h(k+1)}\left(1+\frac{\varphi^{'}}{\varphi}h\right)^{'}+\frac{2}{k+1}(f^{-1})_{\iota}f\frac{G_{\iota}}{G}+\frac{G_{\iota}h_{\iota}}{(k+1)Gh}\left(1+\frac{\varphi^{'}}{\varphi}h\right)\\
&\quad +\frac{\varphi^{'}}{(k+1)\varphi G}h_{\iota}G_{\iota}+\frac{h_{\iota}}{(k+1)h G}G_{\iota}+\frac{G_{\iota\iota}}{(k+1)G}-\frac{k}{(k+1)^{2}}f^{2}(f^{-1})^{2}_{\iota}-\frac{k}{(k+1)^{2}}\frac{h^{2}_{\iota}}{h^{2}}\left(1+\frac{\varphi^{'}}{\varphi}h\right)^{2}\\
&\quad -\frac{2kf(f^{-1})_{\iota}h_{\iota}}{h(k+1)^{2}}\left(1+\frac{\varphi^{'}}{\varphi}h\right)-\frac{k}{(k+1)^{2}}\frac{G^{2}_{\iota}}{G^{2}}-2\frac{k}{k+1}\frac{f}{G}(f^{-1})_{\iota}G_{\iota}-\frac{2k}{(k+1)^{2}}\frac{h_{\iota}G_{\iota}}{hG}\left(1+\frac{\varphi^{'}}{\varphi}h\right)\\
&=1+\frac{1}{k+1}f(f^{-1})_{\iota\iota}+\frac{2f}{(k+1)^{2}h}(f^{-1})_{\iota}h_{\iota}\left(1+\frac{\varphi^{'}h}{\varphi}\right)+\frac{h_{\iota\iota}}{(k+1)h}\left(1+\frac{\varphi^{'}h}{\varphi}\right)\\
&\quad +\frac{h^{2}_{\iota}}{(k+1)h}\left(1+\frac{\varphi^{'}h}{\varphi}\right)^{'}-\frac{k}{(k+1)^{2}}f^{2}(f^{-1})^{2}_{\iota}+\frac{h^{2}_{\iota}}{(k+1)^{2}h^{2}}\left(1+\frac{\varphi^{'}h}{\varphi}\right)\left(\frac{\varphi^{'}h}{\varphi}-k\right)\\
&\quad +\frac{2}{k+1}(f^{-1})_{\iota}f\frac{G_{\iota}}{G}+\frac{G_{\iota}h_{\iota}}{(k+1)Gh}\left(1+\frac{\varphi^{'}}{\varphi}h\right)+\frac{\varphi^{'}}{(k+1)\varphi G}h_{\iota}G_{\iota}+\frac{h_{\iota}}{(k+1)h G}G_{\iota}+\frac{G_{\iota\iota}}{(k+1)G}\\
&\quad-\frac{k}{(k+1)^{2}}\frac{G^{2}_{\iota}}{G^{2}}-2\frac{k}{k+1}\frac{f}{G}(f^{-1})_{\iota}G_{\iota}-\frac{2k}{(k+1)^{2}}\frac{h_{\iota}G_{\iota}}{hG}\left(1+\frac{\varphi^{'}}{\varphi}h\right)\\
&=\frac{1+\frac{\varphi^{'}h}{\varphi}}{(k+1)}\frac{h_{\iota\iota}+h}{h}+\frac{h^{2}_{\iota}}{(k+1)h}\left(1+\frac{\varphi^{'}h}{\varphi}\right)^{'}\\
&\quad-\frac{1+\frac{\varphi^{'}h}{\varphi}}{h(k+1)^{2}}f \left[ h_{\iota}\left(\frac{k-\frac{\varphi^{'}h}{\varphi}}{fh}  \right)^{\frac{1}{2}} -(f^{-1})_{\iota}\left(\frac{hf}{k-\frac{\varphi^{'}h}{\varphi}} \right)^{\frac{1}{2} }\right]^{2}\\
&\quad +\frac{1}{k+1}\left[\left(k- \frac{\varphi^{'}h}{\varphi} \right)-(f^{-1})^{2}_{\iota}f^{2}\left( \frac{k}{k+1}+\frac{1}{k+1}\frac{1+\frac{\varphi^{'}h}{\varphi}}{\frac{\varphi^{'}h}{\varphi}-k}\right)+(f^{-1})_{\iota\iota}f\right]\\
&\quad +\frac{2}{k+1}(f^{-1})_{\iota}f\frac{G_{\iota}}{G}+\frac{G_{\iota}h_{\iota}}{(k+1)Gh}\left(1+\frac{\varphi^{'}}{\varphi}h\right)+\frac{\varphi^{'}}{(k+1)\varphi G}h_{\iota}G_{\iota}+\frac{h_{\iota}}{(k+1)h G}G_{\iota}+\frac{G_{\iota\iota}}{(k+1)G}\\
&\quad-\frac{k}{(k+1)^{2}}\frac{G^{2}_{\iota}}{G^{2}}-2\frac{k}{k+1}\frac{f}{G}(f^{-1})_{\iota}G_{\iota}-\frac{2k}{(k+1)^{2}}\frac{h_{\iota}G_{\iota}}{hG}\left(1+\frac{\varphi^{'}}{\varphi}h\right).
\end{split}
\end{equation}
Making use of the assumptions of Theorem \ref{main*}, one see that $(1+\frac{\varphi^{'}}{\varphi}h)^{'}\geq 0$, and $\frac{\varphi^{'}}{\varphi}h\leq-1$. \eqref{Gd} becomes
\begin{equation}
\begin{split}
\label{fi}
&1+\Theta^{-\frac{1}{k+1}}(\Theta^{\frac{1}{k+1}})_{\iota \iota}\\
&\geq \frac{1+\frac{\varphi^{'}h}{h}}{k+1}\frac{h_{\iota\iota}+h}{h}
 +\frac{1}{k+1}\left[\left(k- \frac{\varphi^{'}h}{\varphi} \right)-(f^{-1})^{2}_{\iota}f^{2}\left( \frac{k}{k+1}+\frac{1}{k+1}\frac{1+\frac{\varphi^{'}h}{\varphi}}{\frac{\varphi^{'}h}{\varphi}-k}\right)+(f^{-1})_{\iota\iota}f\right]\\
&\quad +\frac{2}{k+1}(f^{-1})_{\iota}f\frac{G_{\iota}}{G}+\frac{G_{\iota}h_{\iota}}{(k+1)Gh}\left(1+\frac{\varphi^{'}}{\varphi}h\right)+\frac{\varphi^{'}}{(k+1)\varphi G}h_{\iota}G_{\iota}+\frac{h_{\iota}}{(k+1)h G}G_{\iota}+\frac{G_{\iota\iota}}{(k+1)G}\\
&\quad-\frac{k}{(k+1)^{2}}\frac{G^{2}_{\iota}}{G^{2}}-2\frac{k}{k+1}\frac{f}{G}(f^{-1})_{\iota}G_{\iota}-\frac{2k}{(k+1)^{2}}\frac{h_{\iota}G_{\iota}}{hG}\left(1+\frac{\varphi^{'}}{\varphi}h\right)\\
&\geq \frac{1+\frac{\varphi^{'}h}{h}}{k+1}\frac{h_{\iota\iota}+h}{h}+\frac{1}{k+1}\left[\left(k- \frac{\varphi^{'}h}{\varphi} \right)-(f^{-1})^{2}_{\iota}f^{2}\left( \frac{k-1-\frac{\varphi^{'}h}{\varphi}}{k-\frac{\varphi^{'}h}{\varphi}}\right)+(f^{-1})_{\iota\iota}f\right]\\
&\quad +\frac{2}{k+1}(f^{-1})_{\iota}f\frac{G_{\iota}}{G}+\frac{G_{\iota}h_{\iota}}{(k+1)Gh}\left(1+\frac{\varphi^{'}}{\varphi}h\right)+\frac{\varphi^{'}}{(k+1)\varphi G}h_{\iota}G_{\iota}+\frac{h_{\iota}}{(k+1)h G}G_{\iota}+\frac{G_{\iota\iota}}{(k+1)G}\\
&\quad-\frac{k}{(k+1)^{2}}\frac{G^{2}_{\iota}}{G^{2}}-2\frac{k}{k+1}\frac{f}{G}(f^{-1})_{\iota}G_{\iota}-\frac{2k}{(k+1)^{2}}\frac{h_{\iota}G_{\iota}}{hG}\left(1+\frac{\varphi^{'}}{\varphi}h\right).
\end{split}
\end{equation}
 Using again the assumptions of Theorem \ref{main*}, explicitly,
\begin{equation}
\begin{split}
\label{c0}
&\left(k- \frac{\varphi^{'}h}{\varphi} \right)-(f^{-1})^{2}_{\iota}f^{2}\left( \frac{k-1-\frac{\varphi^{'}h}{\varphi}}{k-\frac{\varphi^{'}h}{\varphi}}\right)+(f^{-1})_{\iota\iota}f\\
&\geq (k+1)-(f^{-1})^{2}_{\iota}f^{2}\frac{k+\vartheta-1}{k+\vartheta}+(f^{-1})_{\iota\iota}f\\
&=(k+1)+(k+\vartheta)f^{\frac{1}{k+\vartheta}}(f^{-\frac{1}{k+\vartheta}})_{\iota\iota}\\
&=f^{\frac{1}{k+\vartheta}}\left[(k+1)f^{-\frac{1}{k+\vartheta}}+(k+\vartheta)(f^{-\frac{1}{k+\vartheta}})_{\iota\iota}\right]\\
&\geq c_{0},
\end{split}
\end{equation}
where $c_{0}$ is a positive constant depending on $f$ and the minimum eigenvalue of $(k+1)f^{-\frac{1}{k+\vartheta}}+(k+\vartheta)(f^\frac{1}{k+\vartheta})_{\iota\iota}$.
On the other hand,
\begin{equation}\label{gi}
G_{\iota}=(\nabla G\cdot e_{\iota})w_{\iota\iota},
\end{equation}
and
\begin{equation}\label{gi2}
G_{\iota\iota}=((\nabla^{2}G\cdot e_{\iota})\cdot e_{\iota})w^{2}_{\iota\iota}-(\nabla G\cdot x)w_{\iota\iota}+(\nabla G\cdot e_{i})w_{i\iota\iota}.
\end{equation}
Using $C^{0}$, $C^{1}$ estimates, then \eqref{gi} and \eqref{gi2} imply that
\begin{equation}
\begin{split}
\label{wii}
&\frac{2}{k+1}(f^{-1})_{\iota}f\frac{G_{\iota}}{G}+\frac{G_{\iota}h_{\iota}}{(k+1)Gh}\left(1+\frac{\varphi^{'}}{\varphi}h\right)+\frac{\varphi^{'}}{(k+1)\varphi G}h_{\iota}G_{\iota}+\frac{h_{\iota}}{(k+1)h G}G_{\iota}+\frac{G_{\iota\iota}}{(k+1)G}\\
&-\frac{k}{(k+1)^{2}}\frac{G^{2}_{\iota}}{G^{2}}-2\frac{k}{k+1}\frac{f}{G}(f^{-1})_{\iota}G_{\iota}-\frac{2k}{(k+1)^{2}}\frac{h_{\iota}G_{\iota}}{hG}\left(1+\frac{\varphi^{'}}{\varphi}h\right)\geq -c_{1}w_{\iota\iota}-c_{2}w^{2}_{\iota\iota}+\frac{1}{k+1}\frac{\nabla G\cdot e_{i}}{G}w_{i\iota\iota}.
\end{split}
\end{equation}
Substituting \eqref{wii} and \eqref{c0} into \eqref{fi}, we get
\begin{equation}
\begin{split}
\label{wt}
\partial_{t}\frac{w^{11}}{h}\leq -\frac{(w^{11})^{2}}{h}\sigma_{k}\Theta\left(c_{0} -c_{1} w_{11}-c_{2}w^{2}_{11}+\frac{2-\vartheta-k}{h} w_{11}+\frac{\nabla G\cdot e_{i}}{G}w_{11i}\right)+\frac{2w^{11}}{h}.
\end{split}
\end{equation}
Using \eqref{hi} into \eqref{wt}, we have
\begin{equation}\label{}
\partial_{t}\frac{w^{11}}{h}\leq -C_{1}\left(\frac{w^{11}}{h}\right)^{2}+C_{2}\frac{w^{11}}{h},
\end{equation}
where $C_{1}$ and $C_{2}$ are positive constants independent of $t$. This give the upper bound of $w^{11}$. The proof is completed.
\end{proof}

\end{proof}
For the particular case $k=n-1$, the upper bound of the principal curvatures is as follows.
\begin{lem}\label{Gau}
Supposing $k= n-1$. Let $f$ be a smooth and positive function on $\sn$, $M_{t}$ be a smooth and strictly convex solution satisfying the flow \eqref{xOrflow}, and let $\varphi, G$ be smooth functions satisfying ${\bf A}$. Then
\[
\kappa_{i}(x,t)\leq C, \quad  \forall(x,t)\in \sn\times(0,\infty)
\]
where $C$ is a positive constant, independent of $t$.

\end{lem}
\begin{proof}
We consider the auxiliary function,
\begin{equation}\label{Emax}
\widetilde{E}(x,t)=\log \lambda_{max}(\{w^{ij}\})-d\log h+\frac{l}{2}\rho^{2},
\end{equation}
where $d$ and $l$ are positive constants to be specified later, and $\lambda_{max}(\{w^{ij}\})$ is the maximal eigenvalue of $\{w^{ij}\}$.

For any fixed $t\in (0,\infty)$, we assume that the maximum of $\widetilde{E}(x,t)$ is attained at $x_{0}$ on ${\sn}$. By a rotation of coordinates, we may assume that $\{w^{ij}(x_{0},t)\}$ is diagonal, and $\lambda_{max}(\{w^{ij}(x_{0},t)\})=w^{11}(x_{0},t)$. Then, \eqref{Emax} turns into
\begin{equation*}\label{E}
E(x,t)=\log w^{11}-d\log h+\frac{l}{2}\rho^{2}.
\end{equation*}
At $x_{0}$, we have
\begin{align}\label{E1}
0=\nabla_{i}E&=-w^{11}\nabla_{i}w_{11}-d\frac{h_{i}}{h}+l\rho\rho_{i}\notag\\
&=-w^{11}(h_{i11}+h_{1}\delta_{1i})-d\frac{h_{i}}{h}+l\rho\rho_{i},
\end{align}
and
\begin{align}\label{E2}
0\geq \nabla_{ii}E&=-w^{11}\nabla_{ii}w_{11}+2w^{11}\sum w^{kk}(\nabla_{1}w_{ik})^{2}-(w^{11})^{2}(\nabla_{i}w_{11})^{2}-d\left(\frac{h_{ii}}{h}-\frac{h^{2}_{i}}{h^{2}}\right)+l\rho^{2}_{i}+l\rho \rho_{ii}.
\end{align}
In addition,
\begin{equation}\label{ED1}
\partial_{t}E=-w^{11}\partial_{t}w_{11}-d\frac{h_{t}}{h}+l\rho\rho_{t}=-w^{11}(h_{11t}+h_{t})-d\frac{h_{t}}{h}+l\rho\rho_{t}.
\end{equation}
Now, recall that
\begin{equation}\label{ED11}
\log(h_{t}+ h)=\log \sigma_{n-1}+\chi(x,t),
\end{equation}
where
\begin{equation*}\label{ED2}
\chi(x,t)=\log\left(\frac{1}{f}hG(\nabla h)\varphi(h)\right).
\end{equation*}
Differentiating \eqref{ED11} gives
\begin{align}\label{ED3}
\frac{h_{jt}+ h_{j}}{h_{t}+ h}&= w^{ik}\nabla_{j}w_{ik}+\nabla_{j}\chi\notag\\
&= w^{ii}(h_{jii}+h_{i}\delta_{ij})+\nabla_{j}\chi,
\end{align}
and
\begin{align}\label{ED4}
&\frac{h_{11t}+ h_{11}}{h_{t}+h}-\frac{( h_{1}+h_{1t})^{2}}{(h_{t}+ h)^{2}}\notag\\
&=\sum w^{ii}\nabla_{11}w_{ii}-\sum w^{ii}w^{kk}(\nabla_{1}w_{ik})^{2}+\nabla_{11}\chi.
\end{align}
The Ricci identity on sphere shows
\[
\nabla_{11}w_{ij}=\nabla_{ij}w_{11}-\delta_{ij}w_{11}+\delta_{11}w_{ij}-\delta_{1i}w_{1j}+\delta_{1j}w_{1i}.
\]
Hence, apply the Ricci identity, \eqref{E2}, \eqref{ED1}, \eqref{ED3} and \eqref{ED4}, at $x_{0}$, we obtain
\begin{align}\label{EFin}
\frac{\partial_{t}E}{h_{t}+ h}&=\frac{-w^{11}(h_{11t}+h_{t})}{h_{t}+ h}-d\frac{h_{t}}{h(h_{t}+ h)}+l\frac{\rho \rho_{t}}{(h_{t}+ h)}\notag\\
&=-w^{11}\left[\frac{h_{11t}+ h_{11}- h_{11}-h+h+h_{t}}{h_{t}+ h}\right]-d\frac{h_{t}}{h(h_{t}+ h)}+l\frac{\rho \rho_{t}}{(h_{t}+ h)}\notag\\
&=-w^{11}\frac{h_{11t}+ h_{11}}{h_{t}+ h}+\frac{1}{h_{t}+1 h}-w^{11}-\frac{d}{h}+\frac{d}{ h +h_{t}}+l\frac{\rho \rho_{t}}{(h_{t}+ h)}\notag\\
&=-w^{11}\frac{h_{11t}+ h_{11}}{h_{t}+h}+\frac{(1+d)}{h_{t}+ h}-w^{11}-\frac{d}{h}+l\frac{\rho \rho_{t}}{(h_{t}+ h)}\notag\\
&\leq-w^{11}\sum w^{ii}\nabla_{11}w_{ii}+w^{11}\sum w^{ii}w^{kk}(\nabla_{1}w_{ik})^{2}-w^{11}\nabla_{11}\chi+\frac{(1+d)}{h_{t}+ h}+l\frac{\rho \rho_{t}}{(h_{t}+ h)}\notag\\
&=-w^{11}w^{ii}(\nabla_{ii}w_{11}-w_{11}+w_{ii})+w^{11}\sum w^{ii}w^{jj}(\nabla_{1}w_{ij})^{2}\notag\\
&\quad -w^{11} \nabla_{11}\chi+\frac{(1+d)}{h_{t}+ h}+\frac{l\rho \rho_{t}}{(h_{t}+ h)}\notag\\
&\leq \sum w^{ii}(w^{11})^{2}(\nabla_{i}w_{11})^{2}-2w^{11}\sum w^{ii}w^{kk}(\nabla_{1}w_{ik})^{2}+\sum w^{ii} d\left(\frac{h_{ii}}{h}-\frac{h^{2}_{i}}{h^{2}}\right)-\sum w^{ii}l\rho^{2}_{i}-\sum w^{ii}l\rho \rho_{ii}\notag\\
&\quad +w^{11}\sum w^{ii}w^{kk}(\nabla_{1}w_{ik})^{2}-w^{11}\nabla_{11}\chi+\frac{(1+d)}{h_{t}+ h}+\frac{l\rho \rho_{t}}{(h_{t}+ h)}+\sum w^{ii}-(n-1)w^{11}\notag\\
&\leq -d\sum w^{ii}+\frac{(n-1)d}{h}-w^{11}\nabla_{11}\chi+\frac{(1+d)}{h_{t}+ h}+l\left[\frac{\rho \rho_{t}}{(h_{t}+ h)}-\sum w^{ii}(\rho^{2}_{i}+\rho \rho_{ii})\right].
\end{align}
It is simple to calculate
\begin{equation}\label{simpl}
\begin{split}
\rho_{t}&=\frac{h h_{t}+\sum h_{k}h_{kt}}{\rho},
\\
\rho_{i}&=\frac{hh_{i}+\sum h_{k}h_{ki}}{\rho}=\frac{h_{i}w_{ii}}{\rho},\\
\rho_{ij}&=\frac{hh_{ij}+h_{i}h_{j}+\sum h_{k}h_{kij}+\sum h_{ki}h_{kj}}{\rho}-\frac{h_{i}h_{j}w_{ii}w_{jj}}{\rho^{3}}.
\end{split}
\end{equation}
Applying \eqref{simpl}, we have
\begin{align}\label{EFin2}
\frac{\rho \rho_{t}}{h_{t}+ h}-\sum w^{ii}(\rho^{2}_{i}+\rho \rho_{ii})&=\frac{hh_{t}}{h_{t}+ h}-h\sum w^{ii}h_{ii}-w^{ii}\sum h^{2}_{ii}\notag\\
&\quad -\frac{|\nabla_{\sn} h|^{2}}{h_{t}+ h}+\sum h_{k}\nabla_{k}\chi\notag\\
&=h-\frac{ \rho^{2}}{h_{t}+ h}+(n-1)h-\sum w_{ii}+\sum h_{k}\nabla_{k}\chi\notag\\
&\leq C-\frac{ \rho^{2}}{h_{t}+h}-\sum w_{ii}+\sum h_{k}\nabla_{k}\chi.
\end{align}
Substituting \eqref{EFin2} into \eqref{EFin}, we obtain
\begin{align}\label{xiaji1}
&\frac{\partial_{t}E}{h_{t}+ h}\leq -d\sum w^{ii}+C(d+l)+\frac{(1+d-l\rho^{2})}{h_{t}+h}-l\sum w_{ii}-w^{11}\nabla_{11}\chi+l\sum h_{k}\nabla_{k}\chi.
\end{align}
Since
\[
\nabla_{k}\chi=\frac{-f_{k}}{f}+\frac{h_{k}}{h}+\frac{\varphi^{'}}{\varphi}h_{k}+\frac{(\nabla G\cdot e_{i})w_{ki}}{G},
\]
and
\begin{align*}\label{}
\nabla_{11}\chi&=\frac{-ff_{11}+f^{2}_{1}}{f^{2}}+\frac{hh_{11}-h^{2}_{1}}{h^{2}}+\frac{\varphi^{''}h^{2}_{1}+\varphi^{'}h_{11}}{\varphi}-\frac{(\varphi^{'})^{2}h^{2}_{1}}{\varphi^{2}}\notag\\
&\quad -\frac{[(\nabla G \cdot e_{1})w_{11}]^{2}}{G^{2}}+\frac{((\nabla^{2}G\cdot e_{1})\cdot e_{1})w^{2}_{11}}{G}+\frac{(\nabla {G}\cdot e_{i})w_{11i}}{G}-\frac{(\nabla G \cdot x)w_{11}}{G}.
\end{align*}
It is clear to see that
\begin{align}\label{hk}
l\sum h_{k}\nabla_{k}\chi&=l\sum h_{k}\left[\frac{-f_{k}}{f}+\frac{h_{k}}{h}+\frac{\varphi^{'}}{\varphi}h_{k}+\frac{(\nabla G\cdot e_{i})\cdot w_{ki}}{G}\right]\notag\\
&\leq C_{1}l+l\frac{\nabla G\cdot e_{k}}{G} h_{k}w_{kk}=C_{1}l+l\frac{\nabla G\cdot e_{k}}{G} \rho \rho_{k}\leq C_{2}l,
\end{align}
and
\begin{align}\label{w111}
-w^{11}\nabla_{11}\chi&=-w^{11}\left(\frac{-ff_{11}+f^{2}_{1}}{f^{2}}+\frac{hh_{11}-h^{2}_{1}}{h^{2}}+\frac{\varphi^{''}h^{2}_{1}+\varphi^{'}h_{11}}{\varphi}-\frac{(\varphi^{'})^{2}h^{2}_{1}}{\varphi^{2}}\right)\notag\\
&\quad+\frac{[(\nabla G \cdot e_{1})]^{2}w_{11}}{G^{2}}-\frac{((\nabla^{2}G\cdot e_{1})\cdot e_{1})w_{11}}{G}-\frac{(\nabla {G}\cdot e_{i})w^{11}w_{11i}}{G}+\frac{\nabla G \cdot x}{G}.
\end{align}
From \eqref{E1} and \eqref{simpl}, we get
\[
w^{11}w_{i11}=-d\frac{h_{i}}{h}+l\rho \rho_{i}.
\]
Hence, \eqref{w111} becomes
\begin{equation}\label{w2}
-w^{11}\nabla_{11}\chi\leq C_{2}w^{11}+C_{3}+C_{4}w_{11}+C_{5}l +C_{6}d.
\end{equation}
Combining \eqref{hk} and \eqref{w2}, we obtain
\begin{equation}
\begin{split}
\label{gaslw9}
&l\sum h_{k}\nabla_{k}\chi-w^{11}\nabla_{11}\chi\\
&\leq \widetilde{C}_{1}l+\widetilde{C}_{2}d+\widetilde{C}_{3}w^{11}+\widetilde{C}_{4}w_{11}+\widetilde{C}_{5}.
\end{split}
\end{equation}
Substituting  \eqref{gaslw9} into \eqref{xiaji1} , if we choose $l\gg d$, then
\begin{align}\label{Finall}
&\frac{\partial_{t}E}{h_{t}+ h}\leq -d\sum w^{ii}+\widetilde{C}(d+l)+\widetilde{C}_{3}w^{11}+\widetilde{C}_{4}w_{11}+\widetilde{C}_{5}l +\widetilde{C}_{6}.
\end{align}
In view of  \eqref{Finall}, if we take
\[
d> \widetilde{C}_{3},
\]
then for $w^{11}$ large enough, there is
\[
\frac{\partial_{t}E}{h_{t}+ h}< 0,
\]
which yields
\[
E(x_{0},t)=\widetilde{E}(x_{0},t)\leq C
\]
for some $C>0$, independent of $t$. This implies that the principal curvature is bounded from above.

\end{proof}

\section{Proof of Theorem \ref{main*} and Theorem \ref{main*2}}
\label{Sec6}
In this section, we complete the main theorems.

{ \bf \emph{ Proof of Theorem \ref{main*}}.} Lemmas \ref{C0}, \ref{C1}, \ref{prin}, \ref{prin2} and \ref{gusk} show that \eqref{xOrflow} is uniformly parabolic in $C^{2}$ norm space. Then, by means of the standard Krylov's regularity theory \cite{K87} of uniform parabolic equation, the estimates of higher derivatives can be naturally obtained, it implies that the long-time existence and regularity of the solution of \eqref{xOrflow}. Moreover, there exists a uniformly positive constant $C$, independent of $t$, such that
\begin{equation}\label{cij}
||h||_{C^{i,j}_{x,t}(\sn\times [0,\infty))}\leq C
\end{equation}
for each pairs of nonnegative integers $i$ and $j$.

With the aid of Arzel\`a-Ascoli theorem and a diagonal argument, we can extract a subsequence $\{t_{s}\}\subset (0,\infty)$ such that there exists a smooth function $\tilde{h}(x)$ satisfying
\begin{equation}\label{}
||h(x,t_{k})-\tilde{h}(x)||_{C^{i}(\sn)}\rightarrow 0
\end{equation}
for each nonnegative integer $i$ as $k\rightarrow \infty$.

Observe that the problem $\frac{1}{f}G\varphi \sigma_{k}=1$ does not have any variational structure, so we may expect the convergence of solutions for all initial hypersurfaces. Following the similar lines in \cite{BIS212}, utilizing the assumptions of Theorem \ref{main*}, as computed in \eqref{inti},
\begin{equation}
\begin{split}
\label{pres}
&\partial_{t}\left(\frac{P}{h}-1\right)-\Theta \sigma^{ij}_{k}\nabla_{ij}\left(\frac{P}{h}-1\right)\\
&=-\frac{P}{h}\left(k+\frac{\varphi^{'}h}{\varphi}+\frac{\nabla G\cdot X}{G}\right)+\frac{P^{2}}{h^{2}}\left(k+\frac{\varphi^{'}h}{\varphi}+\frac{ (\nabla G\cdot hv)}{G}\right)+2\frac{\Theta}{h}\sigma^{ij}_{k}\nabla_{i}h\nabla_{j}\left(\frac{P}{h}-1\right)\\
&=\left(k+\frac{\varphi^{'}h}{\varphi}+\frac{\nabla G\cdot X}{G}\right)\left(\frac{P}{h}-1\right)\frac{P}{h}+2\frac{\Theta}{h}\sigma^{ij}_{k}\nabla_{i}h\nabla_{j}\left(\frac{P}{h}-1\right).
\end{split}
\end{equation}
Since $k+\frac{\varphi^{'}h}{\varphi}+\frac{\nabla G\cdot X}{G}<0$, choosing the initial hypersurface $M_{0}$ satisfying $(\frac{P}{h}-1)_{M_{0}}>0$, using \eqref{pres}, one see that the positivity of $\frac{P}{h}-1$ is preserved along the flow, so we obtain
\begin{equation}\label{}
\partial_{t}h=P-h>0 \quad   {\rm for \ all }\ t\geq 0.
\end{equation}
By Lemma \ref{C0},
\begin{equation}\label{}
C\geq h(x,t)-h(x,0)=\int_{0}^{t}(P-h)(x,t)dt.
\end{equation}
This implies that
\[
\int_{0}^{\infty}(P-h)(x,t)dt\leq C.
\]
Then there exists a subsequence of $t_{k}\rightarrow \infty$ such that
\[
(P-h)(x,t_{k})\rightarrow 0 \ {\rm as} \ t_{k}\rightarrow \infty.
\]
Namely,
\[
\left(
\varphi (h(x,t_{k}))G(X_{t_{k}})\sigma_{k}(x,t_{k})h(x,t_{k})\frac{1}{f(x)}-h(x,t_{k})\right)\rightarrow 0 \ {\rm as} \ t_{k} \rightarrow \infty.
\]
Passing to the limit, we obtain
\[
\varphi (\tilde{h}(x))G(\nabla \tilde{h})\tilde{\sigma}_{k}(x)\frac{1}{f(x)}=1.
\]
Thus the hypersurface with support function $\tilde{h}(x)$ is our desired solution to
\[
\varphi(h(x))G(\nabla h)\sigma_{k}(x)\frac{1}{f(x)}=1.
\]

{ \bf \emph{ Proof of Theorem \ref{main*2}}.} Similarly, Lemmas \ref{C0}, \ref{C1}, \ref{prin}, \ref{prin2}, \ref{Gau} reveal that \eqref{cij} holds. We are in a position to establish the functional (see also \cite{LY20}) corresponding  to \eqref{xOrflow} with $k=n-1$ as
\begin{equation}\label{Jt}
J(t)=\int_{\sn}\left(\int^{t}_{0}\frac{1}{\varphi(s)}ds\right)f(x)dx-\int_{\sn}\left(\int^{\rho}_{0}G(s,u)s^{n-1}ds\right)du.
\end{equation}
Under the flow \eqref{xOrflow}, the monotonicity of $J(t)$ is showed as in the following
\begin{equation}
\begin{split}
\label{monj}
\frac{d J(t)}{dt}&=\int_{\sn}\frac{f(x)}{\varphi(h)}\frac{\partial h}{\partial t}dx-\int_{\sn}G(\rho,u)\rho^{n-1}\frac{\partial \rho}{\partial t}du\\
&=\int_{\sn}\frac{f(x)}{\varphi(h)}\frac{\partial h}{\partial t}dx-\int_{\sn}G(\rho,u)\frac{\partial h}{\partial t}\frac{\rho^{n}}{h}du\\
&=\int_{\sn}\frac{\partial h}{\partial t}\left(\frac{f(x)}{\varphi (h)}-G(\rho,u)\sigma_{n-1}(x,t)\right)dx\\
&=-\int_{\sn}\frac{\left(G\sigma_{n-1}(x,t)h\frac{1}{f}-h\right)^{2}}{\frac{1}{f}\varphi h}dx\leq 0.
\end{split}
\end{equation}
Using \eqref{monj} and Lemmas \ref{C0}, \ref{C1}, there is
\begin{equation}\label{}
C\geq J(0)-J(t)=\int^{t}_{0}\int_{\sn}\frac{\left(G\varphi\sigma_{n-1}(x,t)h\frac{1}{f}-h\right)^{2}}{\frac{1}{f}\varphi h}dxdt,
\end{equation}
where $C$ is a positive constant, independent of $t$.
This implies that
\begin{equation}\label{}
\int^{+\infty}_{0}\int_{\sn}\frac{\left(G\varphi\sigma_{n-1}(x,t)h\frac{1}{f}-h\right)^{2}}{\frac{1}{f}\varphi h}dxdt\leq C,
\end{equation}
Then, there exists a subsequence of $t_{k}\rightarrow \infty$ such that
\begin{equation}\label{}
\int_{\sn}\frac{\left(G\varphi\sigma_{n-1}(x,t_{k})h\frac{1}{f}-h\right)^{2}}{\frac{1}{f}\varphi h(x,t_{k})}dx\rightarrow 0, \quad {\rm as}\ t_{k}\rightarrow \infty.
\end{equation}
Taking a limit, we get
\begin{equation}\label{}
\int_{\sn}\frac{\left(G(\nabla \tilde{h})\varphi(\tilde{h})\tilde{\sigma}_{n-1}(x)\tilde{h}\frac{1}{f}-\tilde{h}\right)^{2}}{\frac{1}{f}\varphi \tilde{h}(x)}dx=0.
\end{equation}
This illustrates that
\[
\frac{1}{f(x)}G(\nabla \tilde{h})\varphi(\tilde{h})\tilde{\sigma}_{n-1}(x)=1.
\]

\section{The uniqueness of the solution}
\label{Sec7}
Observe that, for general $\varphi$ and $G$, there is no uniqueness result for solution to \eqref{Or-Mong}. In what follows, we shall give a special uniqueness result for \eqref{Or-Mong}.

\begin{theo}\label{unique}
Suppose $G(y)=G(|y|)$. If
\begin{equation}\label{}
\varphi(ms_{1})G(ms_{2})\leq \varphi(s_{1})G(s_{2})m^{-k}
\end{equation}
holds for some positive $s_{1}, s_{2}$, $m\geq 1$. Then the solution to the equation
\begin{equation}\label{uneq}
 \frac{1}{f(x)}\sigma_{k}(x)\varphi(h)G(\nabla h)=1
\end{equation}
is unique. Here $1\leq k\leq n-1$, $k$ is an integer.
\end{theo}
\begin{proof}
Assume $h_{1}$ and $h_{2}$ be two solutions of equation \eqref{uneq}. To prove $h_{1}=h_{2}$, on the one  hand, we take by contradiction, with satisfying $\max \frac{h_{1}}{h_{2}}>1$. Suppose $\frac{h_{1}}{h_{2}}$ achieves its maximum at $x_{0}\in\sn$. It follows $h_{1}(x_{0})>h_{2}(x_{0})$. Let $\Lambda={\rm log}\frac{h_{1}}{h_{2}}$. So, at $x_{0}$, one has
\begin{equation}\label{fide}
0=\nabla_{\sn} \Lambda=\frac{\nabla_{\sn} h_{1}}{h_{1}}-\frac{\nabla_{\sn} h_{2}}{h_{2}},
\end{equation}
and applying \eqref{fide}, there is
\begin{equation}
\begin{split}
\label{sede}
0&\geq \nabla^{2}_{\sn}\Lambda\\
&=\frac{\nabla^{2}_{\sn}h_{1}}{h_{1}}-\frac{\nabla_{\sn} h_{1}\otimes \nabla_{\sn} h_{1} }{h^{2}_{1}}-\frac{\nabla^{2}_{\sn}h_{2}}{h_{2}}+\frac{\nabla_{\sn} h_{2}\otimes \nabla_{\sn} h_{2} }{h^{2}_{2}}\\
&=\frac{\nabla^{2}_{\sn}h_{1}}{h_{1}}-\frac{\nabla^{2}_{\sn}h_{2}}{h_{2}}.
\end{split}
\end{equation}
Since $h_{1}$ and $h_{2}$ are solutions of equation \eqref{uneq}, using \eqref{uneq} and \eqref{sede}, at $x_{0}$,  we obtain
\begin{equation}
\begin{split}
\label{h1h2}
1&=\frac{\sigma_{k}(\nabla^{2}_{\sn}h_{2}+h_{2}I)G(|\nabla_{\sn}h_{2}+h_{2}x|)\varphi(h_{2})}{\sigma_{k}(\nabla^{2}_{\sn}h_{2}+h_{1}I)G(|\nabla_{\sn}h_{1}+h_{1}x|)\varphi(h_{1})}\\
&=\frac{h^{k}_{2}\sigma_{k}\left(\frac{\nabla^{2}_{\sn}h_{2}}{h_{2}}+I\right)G\left(h_{2}\sqrt{\Big|\frac{\nabla_{\sn}h_{2}}{h_{2}}\Big|^{2}+1}\right)     \varphi(h_{2})}{h^{k}_{1}\sigma_{k}\left(\frac{\nabla^{2}_{\sn}h_{1}}{h_{1}}+I\right)G\left(h_{1}\sqrt{\Big|\frac{\nabla_{\sn}h_{1}}{h_{1}}\Big|^{2}+1}\right)     \varphi(h_{1})}\\
&\geq \frac{h^{k}_{2}G\left(h_{2}\sqrt{\Big|\frac{\nabla_{\sn}h_{1}}{h_{1}}\Big|^{2}+1}\right)     \varphi(h_{2})}{h^{k}_{1}G\left(h_{2}\sqrt{\Big|\frac{\nabla_{\sn}h_{1}}{h_{1}}\Big|^{2}+1}\right)     \varphi(h_{1})}
\end{split}
\end{equation}
Set $h_{2}(x_{0})=\delta h_{1}(x_{0})$ and $s=h_{1}\left(\sqrt{\Big|\frac{\nabla_{\sn}h_{1}}{h_{1}}\Big|^{2}+1}  \right)(x_{0})$. Then
\[
\frac{\delta^{k}G(\delta s)\varphi(h_{2})}{G(s)\varphi(h_{1})}\leq 1,
\]
i,e,
\[
G(\delta s)\varphi(\delta h_{1})\leq G(s)\varphi (h_{1})\delta ^{-k}.
\]
In light of the assumption in Theorem \ref{unique}, $\delta \geq 1$, it implies that $h_{1}(x_{0})\leq h_{2}(x_{0})$, which is a contradiction. This reveals
\begin{equation}\label{h1xiao}
\max \frac{h_{1}}{h_{2}}\leq 1.
\end{equation}
On the other hand, interchanging the role of $h_{1}$ and $h_{2}$, applying the same argument as above, we have
\begin{equation}\label{h2xiao}
\max\frac{h_{2}}{h_{1}}\leq 1.
\end{equation}
Combining \eqref{h1xiao} and \eqref{h2xiao}, this illustrates that $h_{1}=h_{2}$. So, we complete the proof.

\end{proof}

\section*{Acknowledgment}The authors would like to express sincere gratitude to their supervisor Prof. Yong Huang for his encouragement and constant guidance.

\end{document}